\documentclass[11pt]{amsart}
\usepackage[english]{babel}
\usepackage{amssymb,amsmath,amsthm}
\usepackage[latin2]{inputenc}
\usepackage{amsfonts, verbatim, amscd}
\parskip=\smallskipamount

\def\ol{\overline}

\newtheorem{theorem}{Theorem}
\newtheorem{proposition}[theorem]{Proposition}
\newtheorem{lemma}[theorem]{Lemma}

\def\acknowledgment{\par\addvspace{17pt}\small\rmfamily
\trivlist\if!\ackname!\item[]\else
\item[\hskip\labelsep
{\bfseries\ackname}]\fi}

\def\C{\mathbb{C}}
\def\R{\mathbb{R}}

\def\N{\mathbb{N}}
\def\D{\mathbb{D}}
\def\B{\mathbb{B}}

\def\cF{\mathcal{F}}

\newcommand{\du}{\mathrm{d}}

\newcommand{\cO}{\mathcal{O}_J}
\newcommand{\clD}{{\overline{\,\D}}}
\newcommand{\Psh}{\mathrm{Psh}_J}
\def\z{\zeta}

\begin{document}
\title{Lelong functional on almost complex manifolds}
\author{Barbara Drinovec Drnov\v sek \& Uro\v s Kuzman}
\address{Faculty of Mathematics and Physics\\ University of Ljubljana\\
Institute of Mathematics, Physics and Mechanics\\Jadranska 19, 1000 Ljubljana, Slovenia}
\email{barbara.drinovec@fmf.uni-lj.si}
\email{uros.kuzman@gmail.com}
\thanks{Research supported by grants P1-0291 and J1-5432, Republic of Slovenia.}

%
%

\subjclass[2000]{32U05, 32U25, 32Q60, 32Q65}
\date{\today}
\keywords{almost complex manifolds, J-holomorphic disc, J-plurisubharmonic function, Lelong functional, disc functional, envelope}

\begin{abstract}
We establish plurisubharmonicity of the envelope
of Lelong functional on almost complex manifolds of real dimension four,
thereby we generalize the corresponding result for complex manifolds.
\end{abstract}

\maketitle

%
%
%
%
%
%

\section{Introduction and the main result}
Let $\D=\{\zeta\in \C\colon |\zeta|< 1\}$ denote the open unit disc in 
the complex plane $\C$. Given a smooth almost complex manifold $(M,J)$ 
we denote by $\cO(\clD,M)$ 
the set of $J$-holomorphic discs in $M$, i.e.,
the set of smooth maps $f\colon \clD \to M$ 
that are $J$-holomorphic in some neighborhood of $\clD$. 
The {\em Lelong functional} associated to a 
nonnegative real function $\alpha$ on $M$ is defined by
\begin{equation}
\label{eq:Lelong}
	L_\alpha(f)= \sum_{\z\in \D} \alpha(f(\z))\,  \log|\z|,
	\qquad f\in\cO(\clD,M).
\end{equation}
Given a $J$-plurisubharmonic function $u\in \Psh(M)$ and a point $p\in M$
we denote by $\nu_u(p) \in [0,+\infty]$ its {\em Lelong number}
at $p$: in any local coordinate system 
$\psi$ on $M$, with $\psi(p)=0$, we have
\[
	\nu_u(p)= \lim_{r\to 0} \frac{\sup_{|\psi(q)|\le r} u(q)}{\log r}.
\]
(We consider the constant function $u=-\infty$ as $J$-plurisubharmonic
and set $\nu_{-\infty}=+\infty$.)
Given a nonnegative function $\alpha\colon M\to \R_+$,  we consider 
the corresponding extremal function on $M$
\begin{equation} \label{eq:maxFalpha}
		v_\alpha=\sup\{u \in \Psh(M)\colon u\le 0,\ \nu_u \ge\alpha\}.
\end{equation}

\begin{theorem}\label{main_thm}
Suppose that $(M,J)$ is a smooth almost complex manifold of real dimension four.
For every nonnegative function $\alpha$ on $M$ the extremal function $v_\alpha$ is $J$-plurisubharmonic and equals the envelope of the Lelong functional on $M$:
$$v_\alpha(p)=\inf\left\{L_\alpha(f)\colon f\in \cO(\clD,M),\, f(0)=p\right\} .$$
\end{theorem} 

The disc formula for the Lelong functional was first obtained in the seminal 
work by Poletsky \cite{Poletsky1993}. 
On manifolds it was proved by L\'arusson and Sigurdsson, 
first for domains in Stein manifolds \cite{Larusson-Sigurdsson1998},
and then, following the work of Rosay \cite{Rosay1,Rosay2},
on all complex manifolds \cite{Larusson-Sigurdsson2003}.
For locally irreducible complex spaces the above result was 
proved recently \cite{BDDFF}.

The generalization of the disc formulas to almost complex manifolds became possible recently: deformation theory of big $J$-holomorphic discs was developed by Sukhov and Tumanov \cite{ST, ST3} and in the case of real dimension $4$ they glued $J$-holomorphic discs to tori \cite{ST2}. We include an overview of these results in section 2. 
%
%
%
%
%
%

\section{On $J$-holomorphic discs}
Let us first recall the basic notations concerning $J$-holomorphic discs. 
Denote by $J_{st}$ the standard almost complex structure corresponding to the multiplication by $i=\sqrt{-1}$ in $\mathbb{R}^{2n}\cong\mathbb{C}^n$. Let $(M,J)$ be an almost complex manifold. By a $\emph{disc in M}$ we mean a 
$\mathcal{C}^1$-map from a neighborhood of the closed unit disc $\overline{\D}\subset\mathbb{C}$ to $M$. The disc $f$ is said to be \emph{J-holomorphic} if 
\[
  \bar{\partial}_J f=\frac{1}{2}\left(\du f - J(f) \circ \du f \circ J_{st}\right)=0.
\] 
Assuming that $J$ is of class $\mathcal{C}^k$, $k\in\N$, it follows that $J$-holomorphic discs are  
of class $\mathcal{C}^{k,\beta}$ for every $0<\beta<1$ (see \cite{IS}). 
We assume for simplicity that the structure $J$ is $\mathcal{C}^{\infty}$-smooth. 

In local coordinates $z\in\mathbb{C}^n$ an almost complex structure $J$ is represented by a $\mathbb{R}$-linear operator satisfying $J(z)^2=-I.$ Assume that $J+J_{st}$ is invertible along the disc. The above condition can be rewritten in the form
\begin{eqnarray}\label{J-hol}f_{\bar{\zeta}}+A(f)\overline{f_{\zeta}}=0,\end{eqnarray}
where $\zeta=x+iy\in\D$ and
$
	A(z)(v)=(J_{st}+J(z))^{-1}(J(z)-J_{st})(\bar{v})
$ 
is a complex linear endomorphism for every $z\in\C^n$. Hence 
$A$ can be considered as a $n\times n$ complex matrix of the same regularity as $J(z)$ acting on $v\in\C^n$. We call $A$ the \emph{complex matrix of} $J$.  

\subsection{Deformations of $J$-holomorphic discs}
Let us introduce the Cauchy-Green operator 
\begin{eqnarray*} T(f)(z)=\frac{1}{\pi}\int_{\D}\frac{f(\zeta)}{z-\zeta}\, \du x \, \du y .\end{eqnarray*} 
We will use two of its classical properties. Firstly, $T$ solves the usual $\bar{\partial}$-equation and secondly, $T$ is a bounded operator mapping the space $\mathcal{C}^{k,\beta}(\overline{\D})$ into the space $\mathcal{C}^{k+1,\beta}(\overline{\D})$ \cite{VEKUA}.
Hence, we can define
\begin{eqnarray}\label{J-hol integralska}
\Phi(f)=f+T\left[A(f)\overline{f_{\zeta}}\right],
\end{eqnarray}
and by (\ref{J-hol}) $\Phi(f)$ is a usual holomorphic vector function. 
Moreover, the operator $\Phi$ maps the space $\mathcal{C}^{k,\beta}(\overline{\D})$ to itself. Let us discuss its invertibility. 

Fix a point $p\in M$. Using dilations one can 
find a local chart centered at $p$ in which the norms of $A$ 
and its derivatives are small (see for instance \cite[Lemma 2.1]{ST3}). 
Thus $\Phi$ is a small perturbation of the identity providing a one-to-one correspondence between the $J$-holomorphic discs and the usual holomorphic discs close to the origin in $\C^n$. This gives the classical Nijenhuis and Woolf theorem stating that given 
a tangent vector $v\in T_p M$ there exists a $J$-holomorphic disc $f$ centered in $p$ and such that $\mathrm{d}f(0)(\partial/\partial\mathrm{Re}\zeta)=\lambda v$ for some $\lambda>0$. 
The disc $f$ depends smoothly on the initial data $(p,v)$ and on the structure $J$.
See \cite{SIKORAV} for a short proof. 

In what follows we will use a slightly modified Cauchy-Green transform. For $b\in\D^*:=\D\backslash\left\{0\right\}$ we define  
\begin{eqnarray}\label{T0a}T_{0,b}(f)(z)=T(f)(z)-T(f)(0)-\frac{z}{b}\left[T(f)(b)-T(f)(0)\right].\end{eqnarray}
Since $T_{0,b}(f)(0)=T_{0,b}(f)(b)=0,$ the operator similar to $\Phi$ fixes the center and one other point of the disc.

\begin{lemma}\label{posledica NW}
Let $(M,J)$ be an almost complex manifold and $p\in M$. There exist a local coordinate system $(U,\psi)$ on $M$ centered at $p$  and
a neighborhood $V$, $p\in V\subset U$,
such that for every $q\in V$ we have a 
$J$-holomorphic disc $f_q\in \mathcal{O}_J(\overline{\D},M)$ with properties $f_q(0)=q$ and $f_q(-|\psi(q)|)=p.$
\end{lemma}

\begin{proof}
Let $\epsilon>0$ and $\B=\left\{p\in\R^{2n}; \; \left|p\right|<2\right\}$. By \cite{ST3} there exist a neighborhood $U$ of $p$ and a smooth chart $\psi\colon U\to \B$ such that $\psi(p)=0$, $\psi^*(J)(p)=J_{st}$ and that $\left\|A\right\|_{\mathcal{C}^1(\overline{\B})}<\epsilon.$ Using a cut off function we extend $\psi^*(J)$ to $\R^{2n}$ so that it equals $J_{st}$ outside a small neighborhood of $\overline{\B}.$
 
We define an operator mapping the space $\mathcal{C}^{1,\beta}(\overline{\D})$ into itself by
\begin{eqnarray}\label{CG modify}
\Phi_{0,b}(f)=f+T_{0,b}\left[A(f)\overline{f_{\zeta}}\right].
\end{eqnarray}
Since $T_{0,b}\colon \mathcal{C}^{k,\beta}(\overline{\D})\to \mathcal{C}^{k+1,\beta}(\overline{\D})$ is uniformly bounded for $\left|b\right|\leq b_0<1,$ we can fix $\epsilon$ small enough so that $\Phi_{0,b}$ is invertible for every $\left|b\right|<1/2.$ 

Let $\phi_q(z)=\psi(q)+z\frac{\psi(q)}{\left|\psi(q)\right|}$ and $b=-|\psi(q)|.$ 
For $\psi(q)$ close enough to the origin (i.e. $q$ close enough to $p$)
we have $\left|b\right|<1/2$ and the disc 
$$f_q(z)=\psi^{-1}\left(\Phi^{-1}_{0,b}(\phi_q)(z)\right)$$
is well defined and $J$-holomorphic. 
\end{proof}

The deformation problem changes when we deal with big $J$-holomorphic discs. 
Locally (\ref{J-hol integralska}) can be viewed as a compact (not small) perturbation of the identity with a Fredholm derivative. In contrast, the 
kernel of its linearization might be non-trivial and hence the Implicit function theorem can not be applied directly. However, $\Phi$ can be made 
invertible by adding a small linear $J_{st}$-holomorphic term. This was done in a recent paper by Sukhov and Tumanov \cite{ST}. Let us briefly explain their
methods. 

Let $f\in \mathcal{O}_J(\overline{\D},M)$. The graph of $f$ is a pseudoholomorphic embedding of $\overline{\mathbb{D}}$
to the manifold $(\mathbb{R}^{2}\times M, J_{st}\otimes J)$. 
Let $\pi_M\colon \R^2\times M\to M$ denote the projection to $M$.
One can find a smooth coordinate map $\Pi(\zeta,q)=(\zeta,z(\zeta,q))$ in the neighborhood of the graph of $f$ 
such that $z(\zeta,f(\zeta))=0$ for $\zeta\in\overline{\D}$
and that the push forward $\tilde{J}=\Pi_*(J_{st}\otimes J)$ is defined on a neighborhood of $\overline{\mathbb{D}}\times \left\{0\right\}\subset \overline{\mathbb{D}}\times \mathbb{C}^{n}$ and $\tilde{J}(\zeta,0)=J_{st}$. Its complex matrix $\tilde{A}$ is of the form 
$$\tilde{A}=\left(\begin{array}{cc}0&0\\ a& A\end{array}\right).$$
A disc $g$ close to $f$ is $J$-holomorphic if and only if its graph is $J_{st}\otimes J$-holomorphic.
The latter is equivalent to $(\zeta,z(\zeta,g(\zeta)))$ being $\tilde{J}$-holomorphic. This holds if and only if $h(\zeta)=z(\zeta,g(\zeta))$ satisfies the equation
\begin{eqnarray}\label{h}h_{\bar{\zeta}}+A(\zeta,h)\overline{h_{\zeta}}+a(\zeta,h)=0.\end{eqnarray}
We seek solutions close to the origin. For a solution $h$ we get the $J$-holomorphic disc $g$ by 
$g(\zeta)=\pi_M(\Pi^{-1}(\zeta,h(\zeta))$, $\zeta\in\overline{\D}$.

We define an operator mapping the space $\mathcal{C}^{k,\beta}(\overline{\D})$ into itself:
\begin{eqnarray}\label{psi}\Psi_{0,b}(h)=h+T_{0,b}\left(A(\zeta,h)\overline{h_{\zeta}}+a(\zeta,h)\right).\end{eqnarray}
Its Fr\'echet derivative at the point $h=0$ is of the form  
\begin{eqnarray}\label{Dphi}
	P_{0,b}(V)=V+T_{0,b}(B_1 V + B_2 \overline{V}),
\end{eqnarray}
where the smooth matrix functions $B_1$ and $B_2$ arise from $A$. There is no $\overline{V_\zeta}$ term since $A(\zeta,0)=0$. The operator $P_{0,b}$ is Fredholm but its kernel might be non-trivial \cite{BOJARSKI}. Hence we need a lemma similar to the \cite[Theorem 3.1]{ST}.  

We introduce the real inner product of vector functions $f=(f_1,\ldots,f_n)$ and $g=(g_1,\ldots,g_n)$:
$$\left\langle f,g\right\rangle=\sum_{j=1}^{n}\int_{\mathbb{D}} f_j\bar{g}_j \mathrm{d}x\mathrm{d}y.$$     

\begin{lemma}\label{popravljen}
Let $N$ be the real dimension of $\ker P_{0,b}$ in the space $\mathcal{C}^{k,\beta}(\overline{\D})$ and $w_1,w_2,\ldots,w_N$ its basis. There exist arbitrarily small polynomial vector functions $p_1,p_2,\ldots,p_N$ such that $p_j(0)=p_j(b)=0$ and that the operator  	
$$\tilde{P}_{0,b}(V):=P_{0,b}(V)+\sum_{j=1}^{N}\left\langle V,w_j\right\rangle p_j $$
is bounded and has trivial kernel. 
\end{lemma}

\begin{proof} 
Let $H_{0,b}$ be the space of holomorphic vector functions $h\in\mathcal{C}^{k,\beta}(\overline{\D})$ such that $h(0)=h(b)=0$. We claim that 
$$H_{0,b}+\mathrm{Range}(P_{0,b})=\mathcal{C}^{k,\beta}(\overline{\D}).$$

Let $V$ be orthogonal to both $H_{0,b}^{\bot}$ and $\mathrm{Range}(P_{0,b})$. The last implies that $V$ is in the kernel of the operator adjoint to $P_{0,b}$. We have
$$P_{0,b}^*(V)(z)=V(z)+\overline{\tilde{B}_1T(\overline{V_{0,b}})}(z)+\tilde{B}_2T(V_{0,b})(z),$$
where $V_{0,b}(\zeta)=\bar{\zeta}(\bar{b}-\bar{\zeta})V(\zeta)$ and $\tilde{B}_j(z)= (z(b-z))^{-1}B_j^T(z)$, $j=1,2$. 

Note that
$W=T(\overline{V_{0,b}})$ satisfies the equation 
$$W_{\bar{\zeta}}+B_1^TW+\frac{\bar{\zeta}(\bar{b}-\bar{\zeta})}{\zeta(b-\zeta)}
\overline{B_2}^{T}\overline{W}=0.$$ 
That is, $W$ is an generalized analytic vector such as studied in \cite{BOJARSKI}. Since 
$V$ is also orthogonal to the space $H_{0,b}$ the function $W$ vanishes on 
$\C\backslash\overline{\D}$. This, similarly to the usual analytic theory, implies that 
$W\equiv 0$ and thus $V\equiv 0$ (see \cite[Corollary 3.4]{ST}). Hence, there exist $p_1, 
p_2,\ldots,p_N\in H_{0,b}$ such that 
\begin{eqnarray}\label{vsota}\mathrm{Span}_{\R}(p_1,p_2,\ldots,p_N)\oplus 
\mathrm{Range}(P_{0,b})=\mathcal{C}^{k,\beta}(\overline{\D}).\end{eqnarray}
Finally, $p_1,p_2,\ldots,p_N$ can be chosen to be polynomials since the linear independence is stable under small perturbations and the range of a Fredholm operator is closed.
\end{proof}

We can now modify the operator ${\Psi}_{0,b}$ defined by (\ref{psi}) 
so that it becomes locally invertible. Indeed, there exist small polynomial functions 
$p_1,p_2,\ldots,p_N$ such that the operator
$$\tilde{\Psi}_{0,b}(h)=\Psi_{0,b}(h)+\sum_{j=1}^{N}\mathrm{Re}(h,w_j)p_j$$
is locally invertible and its inverse $\tilde{\Psi}_{0,b}^{-1}$ is defined for all
$\mathcal{C}^{k,\beta}(\overline{\D})$-small holomorphic vector functions $\phi$. 
We get a $J$-holomorphic disc $g$ close to $f$ by 
$\pi_M(\Pi^{-1}(\zeta,\tilde{\Psi}^{-1}_{0,b}(\phi)(\zeta)))$ where $\phi$
is $J_{st}$-holomorphic. Moreover, $\phi(0)=0$ implies $g(0)=f(0)$, and $\phi(b)=0$
implies $g(b)=f(b)$. 

\begin{lemma}\label{translacija}
Let $p\in M$ and $f_p\in \mathcal{O}_J(\overline{\D},M)$ such that $f_p(0)=p.$ For $b\in \D^*$ there exists a small neighborhood $U$ of $p$ in $M$ 
so that for every $q\in U$ there is a $J$-holomorphic disc $f_q$ centered at $q$ with $f_q(b)=f_p(b).$
In particular, if $\dim_{\R}M=4$ and $f_p$ is a smooth immersion, discs $f_q$ depend smoothly on $q$, 
and the only points where the map $G(q,w)=f_q(w)$, $(q,w)\in U\times\overline{\mathbb{D}}$, is not locally a diffeomorphism, are in $w=b$. 
\end{lemma}

\begin{proof} 
We define $\Pi$ and $\tilde{\Psi}_{0,b}$ in the neighborhood of the graph of $f_p$. 
Take a point $s$ close to the origin in $\C^n$ and define 
$$\phi_{s}(w)= \left(1-\frac{w}{b}\right)s.$$
Let $q=\pi_M(\Pi^{-1}(0,s))$ and  
$f_q(\zeta)=\pi_M(\Pi^{-1}(\zeta,\tilde{\Psi}_{0,b}^{-1}(\phi_s)(\zeta)))$. 
We have $\phi_s(b)=0$, thus $f_q(b)=f_p(b)$. 
Moreover, we have $h_s(0)=s$ which implies $f_q(0)=q$. The rest can be verified by a straightforward computation.
\end{proof}

We end this subsection by proving an advanced form of \cite[Theorem 1.1]{ST}. We show that a $J$-holomorphic disc can be approximated by a $J$-holomorphic immersion fixing the center and one other point. A small change in the proof gives approximation
by immersion fixing any finite number of points.

\begin{proposition}\label{imerzija}
Let $f\in \mathcal{O}_J(\overline{\D},M)$
and $b\in\D^*$. There exists an immersion $g\in \mathcal{O}_J(\overline{\D},M)$ arbitrarily close to $f$ 
such that $f(0)=g(0)$ and $f(b)=g(b).$
\end{proposition}

\begin{proof}
In \cite{ST} the proof is based on the transversality theory. We construct $g$ 
directly; we follow \cite{BDD} where the approximation
with interpolation on a finite set of points by immersions was obtained for 
$J_{st}$-holomorphic discs. The idea of the proof is to push the disc slightly
in an almost orthogonal direction near any critical point; this provides 
regularity near the critical points and if the contribution is  small enough
we get an immersion. 

First we choose a smooth Riemannian metric $\rho$ on $M$.
We denote by $f'=f_\zeta$. We may approximate $f$ by $\zeta\mapsto f(r\zeta)$ for some $r$ close to $1$
and therefore we may assume that all critical points of $f$ lie in $\D$.
Let $a_j\in \D$, $j=1,\ldots,m$, be the critical points $f$, i.e., $f'(a_j)=0$. 
Let $(\psi_j,U)$ be a local chart around $f(a_j)$ such that $\psi_j(a_j)=0$ and 
$\psi_j^*(J)(0)=J_{st}$. By (\ref{J-hol}) there exist $v\in \C^n$ and $k\geq 2$ such that 
$\psi_j\circ f(z)= z^{k}v+O(\left|z\right|^{k})$ \cite{MCDUFF}. 
Hence there is a natural way of determining the tangent space $T_{f(a_j)}f(\overline{\D}).$ 
Let $w_j\in T_{f(a_j)}M$ be orthogonal to this space. Working locally in the coordinates 
$\Pi$ defined in the neighborhood of the graph of $f$ this direction corresponds to some 
vector $(0,v_j)\in\C\times\C^n$. We may choose $v_j$ small.

We claim to have a solution $h$ of (\ref{h}) satisfying $h(0)=h(b)=h(a_j)=0$ and 
$h'(a_j)=v_j$ for $j=1,\ldots, m$. Indeed, we define an operator similar to the one in 
(\ref{psi}) using a different normalization of the Cauchy-Green transform $T$. For a 
vector valued function $f$ we set $P_f$ to be a polynomial satisfying $P_f(0)=T(f)(0),$ 
$P_f(b)=T(f)(b)$, $\ P_f(a_j)=T(f)(a_j)$ and $P_f'(a_j)=T(f)'(a_j).$ 
We define $T_P(f)(z)=T(f)(z)-P_f(z)$ and 
$$\Psi_P(h)=h+T_P(A(h)\overline{h_\zeta}).$$
Again, $\Psi_P$ has to be modified into $\tilde{\Psi}_P$ in order become invertible, that is, a result similar to Lemma \ref{popravljen} has to be proved. After that the problem 
becomes a matter of finding a $J_{st}$-holomorphic polynomial $\phi_h=\tilde{\Psi}_P(h)$ 
with the same interpolation properties as $h$. We leave the details to the reader. 

Let $g_1(\zeta)=\pi_M(\Pi^{-1}(\zeta,\tilde{\Psi}_P^{-1}(\phi_h)(\zeta))$, $\zeta\in  \overline{\D}$,
and by the above we have $g_1'(\zeta)=w_j$.
We fix conic neighborhoods $\mathcal{U}_j$ and $\mathcal{V}_j$ of 
$T_{f(a_j)}f(\overline{\D})$  and of $w_j$ in $T_{f(a_j)}M$, respectively,
such that every $u\in\mathcal{U}_j$ and $v\in\mathcal{V}_j$ are linearly independent. 
Since $TM$ is locally trivial there is a  neighborhood $S_j$ of $f(a_j)$ 
such that for $p\in S_j$ we have (by abuse of notation) $\mathcal{U}_j\cup\mathcal{V}_j\subset T_pM$.
There exists $\delta>0$ such that $\left\|w_j-v\right\|_\rho<\delta$ implies $v\in \mathcal{V}_j$. 
Moreover, we can choose a neighborhood $U_j$ of $a_j$ in $\D$
such that for every $\zeta\in U_j$ we have $f(\zeta)\in S_j$, $f'(\zeta)\in \mathcal{U}_j$ and 
$\left\|w_j-g_1'(\zeta)\right\|_\rho<\delta/2$. 

For $0<\lambda\leq 1$ we define 
$g_\lambda(\zeta)=\pi_M(\Pi^{-1}(\zeta,\tilde{\Psi}_P^{-1}(\lambda \phi_h)(\zeta)))$, $\zeta\in  \overline{\D}$. We have
$$g_\lambda'(\zeta)=f'(\zeta)+\lambda g_1'(\zeta)+ O(\delta,\lambda),\ (\delta,\lambda)\to 0.$$
Hence, for $\zeta\in U_j$ and for small $\lambda$ we have $f'(\zeta)\in\mathcal{U}_j$ and 
$g_\lambda'(\zeta)-f'(\zeta)\in\mathcal{V}_j$, that is, $g_\lambda$ is immersed 
on every $U_j$. Moreover, $g_\lambda$ tends to $f$ when $\lambda$ tends to zero. Hence 
we can set $\lambda$ small enough so that $g=g_\lambda$ is immersed on the complement 
$\overline{\D}\backslash\cup_{j=1}^{m}U_j$ as well.
\end{proof}

\subsection{Attaching discs to a real torus}
We present a geometric construction from \cite{PROPER, ST2} needed in the sequel. Given a real torus we attach a disc to its boundary. We restrict to almost complex manifolds with $\dim_{\R}M=4$. 

\vskip 0.1 cm
\noindent{\bf\emph{Formulation of the problem:}} Let $g$ be a $J$-holomorphic immersion defined on $\D_{\gamma}:=(1+\gamma)\D$ for $\gamma>0$. Given a neighborhood $U$ of $\partial{\D},$ we consider $J$-holomorphic discs $f_z\colon \D_\gamma\to M$ satisfying the condition $f_z(0)=g(z)$ and such that the direction $\mathrm{d}f_z\left(\frac{\partial }{\partial \textrm{Re}(z)}\right)$ is not tangent to $g$. Furthermore, we assume that the map
$$G\colon U\times\D_\gamma \to M, \;\;\; G(z,w)=f_z(w)$$
is smooth, that the winding number of its $w$-component around $\partial \D$ is zero and that the set of points where $G$ is not diffeomorphic is discrete in every fiber $G(z,\cdot)$. We seek a family of $J$-holomorphic discs centered at $g(0)$ and attached to the real torus  $\Lambda=G(\partial\D\times\partial \D).$ 

\vskip 0.2 cm
The first step is to extend the family $G$ to the whole bidisc $\overline{\D}^2$. By the assumption on the winding number the vector field $$X_{g(z)}:=\textrm{d}G_z(z,0)\left(\frac{\partial }{\partial \textrm{Re}(w)}\right)$$ can be extended from $U$ to the whole neighborhood of $\overline{\D}$ as a non-vanishing vector field. Hence we can apply the Nijenhuis-Woolf theorem and construct small discs with centers in the set $\D\backslash U$ \cite[sec. ~5.3]{PROPER}. 

Next, we push forward the structure $\tilde{J}=\mathrm{d}G^{-1}\circ (J)\circ \mathrm{d}G$. Since the set of singular points for $G$ is discrete in every fiber $G(z,\cdot)$ we only have removable singularities and $\tilde{J}$ can be defined on some neighborhood of the closed bidisc in $\mathbb{C}^2$ \cite[Theorem 3.1]{ST2}. The discs $\zeta\mapsto(z(\zeta),w(\zeta))$ are $\tilde{J}$-holomorphic if and only if they solve the following elliptic system 
\begin{equation*}
	z_{\bar{\zeta}}=c(z,w)\overline{z_{\zeta}}, \qquad
  w_{\bar{\zeta}}=d(z,w)\overline{z_{\zeta}}, 
\end{equation*}
where $c,d:\overline{\mathbb{D}}\times (1+\gamma)\overline{\mathbb{D}}\to\mathbb{C},$ $\gamma>0$, are functions of class $\mathcal{C}^{\beta}$ in $z$ uniformly in $w$ for some $0<\beta<1$ and Lipschitz in $w$ uniformly in $z$,
such that $\left|c(z,w)\right|\leq c_0<1$ and $c(z,0)=d(z,0)=0$ 
(see \cite[sec. 4]{ST2}). Thus the construction of the disc attached to $\Lambda$ is now reduced to finding a solution of the above system attached to the boundary of the bidisc $\left|z(\zeta)\right|=\left|w(\zeta)\right|=1$ for $\zeta\in\partial\D$ that stays inside its small neighborhood.
\begin{lemma}
\label{lema_RH}
 For every $n\in\mathbb{N}$ there exists a function pair $(u_{n},v_{n})$, such that
 \begin{itemize}
   \item[\rm (i)] the disc $(z(\zeta), w(\zeta))=(\zeta 
    e^{u_{n}(\zeta)},\zeta^{n}e^{v_{n}(\zeta)})$ is $\tilde{J}$-holomorphic and attached 
    with its boundary to $\partial\mathbb{D}\times\partial\mathbb{D}$,
   \item[\rm (ii)] $\left|\zeta^{n}e^{v_{n}(\zeta)}\right|<1+\gamma$ for 
   $\zeta\in\overline{\D}$ and $n$ big 
    enough,
   \item[\rm (iii)] $\left\|u_{n}\right\|_{\infty}$ tends to zero, 
   when $n$ tends to $\infty$, 
    and
   \item[\rm (iv)] there exists $C>0$ such that 
   $\left\|v_n\right\|_{\mathcal{C}^{\beta}(\overline{\D})}<C$ for every $n\in\N$.
 \end{itemize}   
\end{lemma} 

\noindent In case of smooth $c$ and $d$ parts (i), (ii) and (iii) follow directly from
the proof in \cite[Theorem 4.1.]{PROPER}. Moreover, the existence of a constant $C_n>0$ such that $\left\|v_n\right\|_{\mathcal{C}^{\beta}(\overline{\D})}<C_n$ for a fixed $n\in\N$ is proved. As commented in \cite[p. 405]{ST2} the result also holds under the present assumptions. Moreover, reading the proof carefully an uniform upper bound for $\left\|v_n\right\|_{\mathcal{C}^{\beta}(\overline{\D})}$ can be found. We leave the details to the reader. 

\section{Proof of the main theorem}
Let $|I|$ denote the normalized arc length of an arc $I\subset \partial\D$.

\begin{lemma}\label{resitve enacbe}
Let $I=\cup_{j=1}^m I_j$ be a union of closed pairwise disjoint arcs in the circle $\partial\D$ and $U=\cup_{j=1}^mU_j$ a union of their pairwise disjoint neighborhoods in $\C$. Let $\chi_n\colon U\cap \overline{\D}\to \C$ be an equicontinuous family of  functions satisfying the condition $\delta<\left|\chi_n\right|<1-\delta$ for some $0<\delta<1$ and every $n\in\N$. Then there exists $N\in\N$ such that for every $k\geq N$, $n\in \N$ and $j\in\left\{1,\ldots,m\right\}$ the equation
$z^k=\chi_n(z)$ has at least $k\left|I_j\right|$ solutions in $U_j\cap \overline{\D}$. 
\end{lemma}
\begin{proof}
Since the family of functions $\chi_n$ is equicontinuous we can cover $I$ with a finite number of open balls $V_r\subset U$, $r=1,\ldots, M$, centered at the points $e^{it_r}\in I$ such that for every $z\in V_r\cap \overline{\D}$ and every $n\in\N$ we have
\begin{eqnarray}\label{ocena}\left|\chi_n(z)-\chi_n(e^{it_r})\right|< \frac{\delta}{2}.\end{eqnarray}
Let $\Delta^n_r\Subset \D^*$ be the open disc of radius $\delta/2$ centered at $\chi_n(e^{it_r})$. The map $\D^*\to\D^*$, $z\longmapsto z^k$, is a $k$-fold covering. The preimage of $\Delta^n_r$ is a disjoint union of $k$ simply connected domains $\Delta^{n,k}_{r,l}\Subset \D^*$, $l=1,\ldots,k$. As the point $z$ traces $\partial \Delta^{n,k}_{r,l}$ once in the positive direction, the image point $z^k$ traces $\partial\Delta^n_r$ once in the positive direction. From the estimate $(\ref{ocena})$ we infer that the function $z\mapsto z^k-\chi_n(z)$ has winding number $1$ around $\partial\Delta^{n,k}_{r,l}$, and hence the equation $z^k=\chi_n(z)$ has a solution $z=z^{n,k}_{r,l}$ in  $\Delta^{n,k}_{r,l}\subset V_r\cap \D$. 

As $k\to\infty$, the sets $\Delta^{n,k}_{r,l}$ converge to the circle $\partial \D$ and are equidistributed around $\partial \D$.
Moreover, an uniform bound (depending only on $\delta$ and $k$) can be found for their diameter and their distance from the boundary. Hence, one can choose $k$ large enough so that  for every $j=1,\ldots,m$ the number of different solutions $z=z^{n,k}_{r,l}$ lying in the union of open balls $V_r$ that cover $I_j$ exceeds the proportional number $k\left|I_j\right|$. \end{proof}

In the proof of the main theorem we use also the envelope of the Poisson functional,
thus let us give here the necessary definitions and results. A {\em disc functional} on $M$ is a function 
\[
	H_M \colon \cO(\clD,M) \to\ol\R= [-\infty,+\infty]. 
\]
Let $p\in M$. The {\em envelope} of $H_M$ is the function $EH_M\colon M\to\overline\R$
defined by 
$$		EH_M(p)= \inf \left\{ H_M(f)\colon f\in \cO(\clD,M),\, f(0)=p\right\}.	$$
Given an upper semicontinuous function 
$u\colon M\to \R\cup\{-\infty\}$, the associated 
{\em Poisson functional}  is defined by
\begin{equation}
\label{eq:Poisson}
	P_u(f)= \frac{1}{2\pi} \int^{2\pi}_0 u(f(e^{i t}))\, {\mathrm d}t,
	\qquad f\in \cO(\clD,M).
\end{equation}
The envelope $EP_u$ of the Poisson functional 
is the largest plurisubharmonic minorant of the 
upper semicontinuous function $u$. In the case when $M=\C^n$ this was proved by 
Poletsky \cite{Poletsky1991,Poletsky1993} 
and by Bu and Schachermayer \cite{Bu-Schachermayer}.
The result was extended to some complex manifolds
by L\'arusson and Sigurdsson 
\cite{Larusson-Sigurdsson1998,Larusson-Sigurdsson2003}, 
and to all complex manifolds by Rosay \cite{Rosay1,Rosay2}.
Using the method of sprays the first named author and Forstneri\v{c}
extended the result to locally irreducible complex spaces \cite{DF2}.
The second named author proved the result in the case of an almost 
complex manifold $M$ of real dimension four \cite{Kuz}.

Now we turn to the proof of Theorem \ref{main_thm}.
Given a nonnegative function $\alpha$ on $M$  let
$$ \cF_\alpha=\{u \in \Psh(M)\colon u\le 0,\ \nu_u \ge\alpha\}.$$
If $u\in \cF_\alpha$ and $f\in \cO(\clD,M)$, then 
$u\circ f\le 0$ is a subharmonic function in a neighborhood of $\clD$ whose 
Lelong number at any point $\z\in \D$ satisfies 
\[
	\nu_{u\circ f}(\z) \ge \alpha(f(\z))\, .
\]
Hence $u\circ f$ is bounded above by the largest subharmonic function 
$v_{\alpha \circ f}\le 0$ on $\D$ satisfying $\nu_{v_{\alpha\circ f}}\ge \alpha\circ f$.
This maximal function $v_{\alpha\circ f}$ is the weighted sum of 
Green functions with coefficients $\alpha\circ f$: 
\[
	v_{\alpha\circ f}(z) = \sum_{\z\in \D} \alpha(f(\z)) \, 
		\log\left| \frac{z-\z}{1-\bar \z z}\right|,
		\quad z\in\D.
\]
(If the sum is divergent then $v\equiv-\infty$.)
Indeed, the difference between $v_{\alpha\circ f}$ and the right hand side above
is subharmonic on $\ol\D$, except perhaps at the points $z$ where 
$\alpha(f(z))>0$; near these points it is bounded above,
so it extends to $\D$ as a subharmonic function.
Since it is clearly $\le 0$ on $\partial\D$, the maximum principle implies that
it is $\le 0$ on all of $\D$ which proves the claim.

Setting $z=0$, we see that for every $u\in \cF_\alpha$ and 
$f\in \cO(\clD,M)$ we have 
\begin{equation} \label{ocena_L_K}
	u(f(0)) 
	\le  \sum_{\z\in \D} \alpha(f(\z)) \log|\z| 
	\le \inf_{\z\in \D} \alpha(f(\z)) \log|\z|.
\end{equation}
The first expression is the Lelong functional (\ref{eq:Lelong}) and the second expression determines the following disc functional on $M$ with values in $[-\infty,0]$:
\begin{equation}
	K_\alpha(f) = \inf_{\z\in \D} \alpha(f(\z)) \log|\z|.
	\label{defK} 
\end{equation}

By taking infima over all analytic discs $f$ centered at $p\in M$ we obtain for any $u\in \cF_\alpha$:
\begin{equation}
\label{eq:ineq-envelopes}
	u(p)\le EL_\alpha(p)
	\le EK_\alpha(p)
	 =: k_\alpha(p).
\end{equation}
The function $k_\alpha\colon M\to [-\infty,0]$,
which is denoted $k^\alpha_M$ in \cite[p.\ 21]{Larusson-Sigurdsson1998},
is related to a certain function studied by Edigarian \cite{Edigarian1997}.

\begin{lemma}
Let $(M,J)$ be an almost complex manifold and $\dim_\R M=4$. Given a nonnegative function $\alpha$ on $M$ the function $k_\alpha$ defined above is upper semicontinuous and
$$EP_{k_\alpha}=v_\alpha .$$
\end{lemma}
\noindent The proof is similar to the proof of \cite[Proposition 5.2]{Larusson-Sigurdsson1998}.
We include it to point out the use of the results from almost complex analysis.

\begin{proof}
Set $p\in M$ and $\epsilon>0$. Since $k_\alpha=EK_\alpha$
there exist $f_p\in \cO(\clD,M)$ with $f_p(0)=p$ and $b\in \D^*$ such that
$$\alpha(f_p(b))\log|b|<k_\alpha(p)+\epsilon.$$
By Lemma \ref{translacija} there is a neighborhood $U$ of $p$ such that
for any $q\in U$ there is a $J$-holomorphic disc $f_q$ centered at $q$ 
with $f_q(b)=f_p(b)$. Therefore for any $q\in U$ we have
$$k_\alpha(q)\le K_\alpha(f_q)\le \alpha(f_q(b))\log|b|<k_\alpha(p)+\epsilon\ ,$$
which proves that $k_\alpha$ is upper semicontinuous.

Since $k_\alpha$ is upper semicontinuous its Poisson envelope $EP_{k_\alpha}$ is the 
largest plurisubharmonic function satisfying $EP_{k_\alpha}\le k_\alpha$.
By (\ref{eq:ineq-envelopes}) we have $u\le k_\alpha$ for all $u\in \cF_\alpha$.
Once we prove that Lelong numbers of $EP_{k_\alpha}$ are bounded bellow by $\alpha$
we get $EP_{k_\alpha}=v_\alpha$.

Choose $p\in M$. By Lemma \ref{posledica NW}
there exist a coordinate system $\psi \colon U\to \R^4$ and an open set $V$, $p\in V\subset U$, 
such that for every $q\in V$ 
we have a $J$-holomorphic disc $f_q$ centered at $q$ and $f(-|\psi(q)|)=p$.
Therefore we have by the same reasoning as above that
$$k_\alpha(q)\le K_\alpha(f_q)\le \alpha(f(-|\psi(q)|))\log|\psi(q)|=
\alpha(p)\log|\psi(q)|$$
for every $q \in V$ which implies
$$\nu_{EP_{k_\alpha}} (p)= \lim_{r\to 0} \frac{\sup_{|\psi(q)|\le r} EP_{k_\alpha}(q)}{\log r}\ge \lim_{r\to 0} \frac{\sup_{|\psi(q)|\le r} k_\alpha(q)}{\log r}\ge \alpha(p).$$
\end{proof}

\begin{proof}[Proof of Theorem 1] 
By the above lemma and by (\ref{eq:ineq-envelopes}) and taking supremum over all $u\in\cF_\alpha$ 
we get $EP_{k_\alpha}=v_\alpha\le E{L}_\alpha$, therefore it suffices to prove that $E{L}_\alpha\leq EP_{k_{\alpha}}$. Equivalently, we need to show that for every continuous function $\varphi\colon M\to \R$ with $\varphi\geq k_\alpha$, immersed $J$-holomorphic disc $g\colon\clD\to M$, and $\epsilon>0$ there exists a $J$-holomorphic disc $f$ such that $f(0)=g(0)$ and
$${L}_{\alpha}(f)=\sum_{z\in\D}\alpha(f(z))\log|z|<\frac{1}{2\pi}\int_0^{2\pi}\varphi\circ g(e^{it})\mathrm{d}t+\epsilon.$$

Since $k_\alpha=EK_{\alpha}$ we have for every fixed $z_0=e^{it_0}\in\partial\D$ a $J$-holomorphic disc $f_{z_0}\colon\clD\to M$ and a point $b_0\in \D^*$ such that $f_{z_0}(0)=g(z_0)=:x_0$ and
$$\alpha(f_{z_0}(b_0))\log|b_0|<\varphi(x_0)+\frac{\epsilon}{3}.$$
We can assume that $f_0$ is immersed (Proposition \ref{imerzija}) and transversal to the initial one denoted by $g$.  
We wish to find a family of such close-to-extremal discs, smoothly parametrized 
by the points in the boundary of $g$. This will be accomplished by using Lemma \ref{translacija}.

There exists a neighborhood $U$ of $z_0$ such that for every point $z\in U$ there exists a $J$-holomorphic disc $f_{z}$ centered at $g(z)$ and satisfying the condition $f_{z}(b_0)=f_{z_0}(b_0)=:y_0.$ By continuity there is a nontrivial closed arc $I'\subset \partial\D$ around the point $z_0$ such that $I'\subset U$ and 
for any closed subarc $I\subset I'$
$$\alpha(y_0)\log|b_0|\cdot |I|<\int_I \varphi\circ g(e^{it})\frac{\textrm{d}t}{2\pi}+\frac{\epsilon}{3}|I|.$$
Repeating this argument at other points of the circle $\partial\D$ we can cover the boundary by closed arcs $I'_1,\ldots,I'_m\subset \partial\D$ with neighborhoods $U'_1,\ldots,U'_m$ such that the two consecutive arcs meet exactly at a common end point and define smooth maps $\tilde{G}_j\colon U'_j\times \overline{\D}\to M$,  
$\tilde{G}_j(z,w)=f_z(w)$, for $j=1\ldots,m,$ with the following properties for $z\in U'_j$:

i) the map $w\mapsto \tilde{G}_j(z,w)$ is $J$-holomorphic and $\tilde{G}_j(z,0)=g(z)$,

ii) $\tilde{G}_j(z,b_j)=:y_j$ for some $b_j\in \D^*$ and for any closed subarc $I_j\subset I'_j$ 
               \begin{eqnarray}\label{ocena1pom}
                 \alpha(y_j) \log|b_j|\cdot|I_j|<\int_{I_j}\varphi\circ                      
                 g(e^{it})\frac{\mathrm{d}t}{2\pi}+\frac{\epsilon}{3}|I_j|.
               \end{eqnarray}
             
\noindent We need to deform the maps $\tilde{G}_j$ on the intersections of their domains
so that they glue into a smooth map $G$ defined on a sufficiently
large subset. 

Let $I'_{m+1}=I'_1$, $U'_{m+1}=U'_1$, and for each $j=1,\ldots,m$ let $p_j$ be the common point of the consecutive arcs $I'_j$ and $I'_{j+1}$. 
On a small neighborhood of $g(p_j)$ the set of small $J$-holomorphic discs is in bijective correspondence with the set of small standard holomorphic discs in $\C^2$ by the classical Nijenhuis-Woolf theorem. Let $V_j$ be the set of all $z$ for which $g(z)$ lies in this neighborhood and such that $p_j\in V_j\subset U'_j\cap U'_{j+1}$. We choose closed pairwise disjoint arcs $I_j\subset I'_j$ and 
pairwise disjoint neighborhoods $U_j$, $I_j\subset U_j\subset U'_j$, such that 
$p_j\notin { \overline{U}_j \cup  \overline{U}_{j+1}} $,
$\partial\D\setminus\cup_{j=1}^m I_j\subset \cup_{j=1}^m V_j$, and
\begin{eqnarray}\label{ocena3pom}
\left|\int_{\partial\D\setminus\cup_{j=1}^m I_j}\varphi\circ g(e^{it})\frac{\mathrm{d}t}{2\pi}\right|<\frac{\epsilon}{3}.
\end{eqnarray}
Furthermore, choose $V_j'\Subset V_j$ slightly smaller open sets and smooth functions $\psi_j\colon U_{j}\to (0,1]$ 
such that $\psi_j(z)=1$ for $z\in U_j \setminus (V_j\cup V_{j-1})$ and $\psi_j(z)$ attains a very small positive value for $z\in (U_j\cap V'_j)\cup (U_j\cap V'_{j-1}).$ We define $G(z,w):=\tilde{G}_j(z,\psi_j(z)w)$ for $z\in U_j$. The discs with centers in $g(U_j\cap V'_j)$ and $g(U_j\cap V'_{j-1})$ are now small. We apply the bijective correspondence defined in the neighborhood of $g(p_j)$ and construct a smoothly varying family of immersed $J$-holomorphic discs with centers in $g(V_j\setminus (U_j\cup U_{j+1}))$. Moreover, we may assume that the winding number of the $w$-component around $\partial \D$ is trivial (otherwise we rotate the small discs). 

We now refer to the subsection 1.2. We extend the family $G$ to some neighborhood of the bidisc $\overline{\mathbb{D}}^2$ and push forward the structure. For $n\in\N$ let $u_{n}$ and $v_{n}$ be defined as in Lemma \ref{lema_RH}. We choose $\eta$, $0<\eta<\epsilon$, such that
\begin{eqnarray}\label{ocena2pom}
                 (\alpha (y_j)+1)\log|b_j|+\frac{\eta}{3}<0  ,\ \
                 j=1,\ldots,m.
               \end{eqnarray}
For every $j=1 , \ldots, m$ we can choose a small open neighborhood
 $\widetilde U_j\Subset U_j$ of the arc $I_j$ such that for every $n\in \N$ and $\zeta\in  \cup_{j=1}^m \widetilde U_j\cap \overline{\D}$
 we have
\begin{equation}
\label{nova}
|e^{-v_n(\zeta)}|<\min\left\{1+\frac{\epsilon}{3(\alpha(y_j)+1)},\frac{1}{\max\left\{|b_1|,\ldots, |b_m|\right\}}\right\} .
\end{equation}
Note that this is possible since by part (iv) of Lemma \ref{lema_RH} the norm $\left\|v_n\right\|_{\mathcal{C}^{\beta}(\overline{\D})}$ is uniformly bounded by a constant and thus there exists $D>0$ such that $\left\|e^{-v_n}\right\|_{\mathcal{C}^{\beta}(\overline{\D})}<D.$ Since $ |e^{-v_n(\zeta)}|=1$ for $\zeta\in\partial\D$ we have  $$|e^{-v_n(\zeta)}|<1+D(1-|\zeta|)^{\beta},\;\;  |\zeta|<1.$$ 
Hence, it suffices to choose the sets $\widetilde U_j$ close enough to the boundary circle. 

Let us define 
$$\chi_n\colon \cup_{j=1}^m \widetilde U_j\cap\overline{\D}\to \C,\;\; \chi_n(\zeta):=b_je^{-v_n(\zeta)}, \;n\in\N.$$ 
By (\ref{nova}) there exists $\delta>0$ such that $\delta<|\chi_n|<1-\delta$ for every $n\in\N$. Furthermore, the family is equicontinuous since the norms $\left\|e^{-v_n}\right\|_{\mathcal{C}^{\beta}(\overline{\D})}$ are uniformly bounded. Hence, by Lemma \ref{resitve enacbe} there exists $N_1\in\N$ such that the equation $\zeta^n=\chi_n(\zeta)$ has at least $n|I_j|$ solutions in every $\widetilde U_j\cap \D$ for $n\geq N_1$.
We denote them by $\zeta_{j,l}$, $l=1,\ldots,l_j$
and
\begin{equation}
 \label{stevilo_resitev}
 l_j\ge n|I_j|,\ \ j=1,\ldots,m ,\, n\geq N_1.
\end{equation}
By part (iii) of Lemma \ref{lema_RH} there exists $N_2\in\N$ such that $\zeta e^{u_n(\zeta)}\in U_j$ for every $\zeta \in \widetilde U_j$ and $n\geq N_2$. Finally, by parts (i) and (ii) of Lemma \ref{lema_RH} we have $N_3\in\N$ such that for $n\geq N_3$ the disc $f\colon\overline{\D}\to M$ given by
$$f(\zeta)=G\left(\zeta e^{u_n(\zeta)},\zeta^n e^{v_n(\zeta)}\right)$$
is well defined and $J$-holomorphic with $f(0)=g(0)$. 

Let us fix now $n$ greater than $\max\left\{N_1,N_2,N_3\right\}.$ We have 
$$f(\zeta_{j,l})=G\left(\zeta_{j,l}\cdot e^{u_n(\zeta_{j,l})},b_j\right)=y_j,\ \
l=1,\ldots, l_j,\, j=1,\ldots, m.$$
Hence
\begin{eqnarray}\label{ocena2}
\sum_{\zeta\in\D}\alpha(f(\zeta))\log|\zeta|\leq \sum_{j=1}^m\alpha(y_j)\sum_{l=1}^{l_j}\log|\zeta_{j,l}|.\end{eqnarray}
Note that by (\ref{nova}), (\ref{stevilo_resitev}) and (\ref{ocena2pom}) one has
$$\sum_{l=1}^{l_j}\log|\zeta_{j,l}|=\frac{1}{n}\sum_{l=1}^{l_j}\left(\log |b_j|+\log|e^{-v_{n}(\zeta_{j,l})}|\right)<|I_j|\log |b_j|+\frac{\epsilon|I_j|}{3(\alpha(y_j)+1)}.$$
Combining this with (\ref{ocena2}), (\ref{ocena1pom}) and (\ref{ocena3pom}) we get
$$\sum_{\zeta\in\D}\alpha(f(\zeta))\log|\zeta|< \sum_{j=1}^m \alpha(y_j)|I_j|\log|b_j|+\frac{\epsilon}{3}<\int_{0}^{2\pi}\varphi\circ g(e^{it})\frac{\mathrm{d}t}{2\pi}+\epsilon.$$
The proof is complete.
\end{proof}


\begin{thebibliography}{4}

\bibitem{BOJARSKI} 
\textsc{B.\ Bojarski}, Theory of a generalized analytic vector (Russian). 
\emph{Ann.\ Pol.\ Math.}, {\bf 17} (1966), 281--320.  

\bibitem{Bu-Schachermayer}
\textsc{S.\ Q.\ Bu,} and \textsc{W.\ Schachermayer},
Approximation of Jensen measures by image measures 
under holomorphic functions and applications.
\emph{Trans.\ Amer.\ Math.\ Soc.}, {\bf 331} (1992), 585--608.



\bibitem{PROPER}
\textsc{B.\ Coupet, A.\ Sukhov,} and \textsc{A.\ Tumanov}, 
Proper J-holomorphic discs in Stein domains of dimension 2. 
\emph{Amer.\ J.\ Math.} {\bf 131} (2009), 653--674.

\bibitem{BDD}
\textsc{B.\ Drinovec Drnov\v sek},
Discs in Stein manifolds containing given discrete sets,
\emph{Math.\ Z.}, {\bf 239} (2002), 683--702.

\bibitem{DF2}
\textsc{B.\ Drinovec Drnov\v sek} and \textsc{F.\ Forstneri\v c},
The Poletsky-Rosay theorem on singular complex spaces,
\emph{Indiana Univ.\ Math.\ J.}  {\bf 61} (2012), 1407-1423.

\bibitem{BDDFF}
\textsc{B.\ Drinovec Drnov\v sek} and \textsc{F.\ Forstneri\v c},
{Disc functionals and Siciak-Zaharyuta extremal functions on singular varieties},
\emph{Ann.\ Polon.\ Math.} {\bf 106} (2012), 171--191. 

\bibitem{Edigarian1997}
\textsc{A.\ Edigarian},
{On definitions of pluricomplex Green function},
\emph{Ann.\ Polon.\ Math.}, {\bf 67} (1997),  233--246. 


\bibitem{IS} 
\textsc{S.\ Ivashkovich} and \textsc{V.\ Shevchisin}, 
Complex Curves in Almost-Complex Manifolds and Meromorphic Hulls, 
\emph{Inst.\ f\"ur Math., Ruhr-Universit\ae t Bochum}, (1999).

\bibitem{Kuz}
\textsc{U.\ Kuzman},
Poletsky theory of discs in almost complex manifolds,
\emph{Complex variables and elliptic equations} to appear, doi: 10.1080/17476933.2012.734300

\bibitem{Larusson-Sigurdsson1998}
\textsc{F.\ L\'arusson} and \textsc{R.\ Sigurdsson}, 
{Plurisubharmonic functions and analytic discs on manifolds},
\emph{J.\ Reine Angew.\ Math.}, {\bf 501} (1998), 1--39. 

\bibitem{Larusson-Sigurdsson2003}
\textsc{F.\ L\'arusson} and \textsc{R.\ Sigurdsson}
{Plurisubharmonicity of envelopes of disc functionals on manifolds},
\emph{J.\ Reine Angew.\ Math.}, {\bf 555} (2003), 27--38. 

\bibitem{MCDUFF}
\textsc{D.\ McDuff} Singularities and positivity of $J$-holomorphic curves.
In `Holomorphic curves in Symplectic geometry', Eds.\ M.\ Audin, J.\ Lafontane, 
Birkhauser (1994), 191--215.

\bibitem{Poletsky1991}
\textsc{E.\ A.\ Poletsky},
{Plurisubharmonic functions as solutions of variational problems},
In: Several complex variables and complex geometry, Part 1 
(Santa Cruz, CA, 1989), 163--171,
Proc.\ Sympos.\ Pure Math., {\bf 52}, Part 1, 
Amer.\ Math.\ Soc., Providence, RI, 1991.

\bibitem{Poletsky1993}
\textsc{E.\ A.\ Poletsky}, 
{Holomorphic currents},
\emph{Indiana Univ.\ Math.\ J.}, {42} (1993), 85--144. 

\bibitem{Rosay1}
\textsc{J.-P.\ Rosay},
{Poletsky theory of disks on holomorphic manifolds},
\emph{Indiana Univ.\ Math.\ J.}, {\bf 52} (2003), 157--169. 

\bibitem{Rosay2}
\textsc{J.-P.\ Rosay},
{Approximation of non-holomorphic maps, and Poletsky theory of discs},
\emph{J.\ Korean Math. Soc.}, {\bf 40} (2003), 423--434.  


\bibitem{SIKORAV}
\textsc{J.\ C.\ Sikorav}, 
Some properties of holomorphic curves in almost complex manifolds. 
In `Holomorphic curves in Symplectic geometry', Eds.\ M.\ Audin, J.\ Lafontane, 
Birkhauser (1994), 165--189.

\bibitem{ST}
\textsc{A.\ Sukhov} and \textsc{A.\ Tumanov}, 
Deformations and transversality of pseudo holomorphic discs, 
\emph{J.\ d'Analyse Math.}  {\bf 116} (2012),  1--16.

\bibitem{ST3} 
\textsc{A.\ Sukhov} and \textsc{A.\ Tumanov}, 
Filling hypersurfaces by discs in almost complex manifolds of dimension 2, 
\emph{Indiana Univ.\ Math.\ J.}, {\bf 57} No.\ 1 (2008), 509--544.

\bibitem{ST2}
\textsc{A.\ Sukhov} and \textsc{A.\ Tumanov}, 
Regularization of almost complex structures and gluing holomorphic discs to tori, 
\emph{Ann.\ Scuola Norm.\ Sup.\ Pisa Cl.\ Sci.} {\bf 5} (2011), 389--411.

\bibitem{VEKUA} 
\textsc{I.\ N.\ Vekua}, 
Generalized analytic functions.
\emph{Pergamon Press, London-Paris-Frankfurt; 
Addison-Wesley Publishing Co., Inc., Reading, Mass.}, 1962. 


\end{thebibliography}
\end{document}